\RequirePackage{etoolbox}
\csdef{input@path}{{style/}{graphics/}}
\makeatletter
\input{arxiv-vmsta.cfg}
\makeatother
\documentclass[numbers,compress]{vmsta}
\usepackage{dcolumn}
\usepackage{enumitem}
\newtheorem{theorem}{Theorem}
\theoremstyle{definition}
\newtheorem{remark}{Remark}
\newtheorem{assumption}{Assumption}

\volume{1}
\pubyear{2014}
\firstpage{181}
\lastpage{194}
\doi{10.15559/15-VMSTA15}

\startlocaldefs

\newcommand{\rrvert}{\vert}
\newcommand{\llvert}{\vert}
\allowdisplaybreaks
\endlocaldefs

\begin{document}
\begin{frontmatter}

\title{Asymptotics for functionals of powers of~a~periodogram}
\author{\inits{L.}\fnm{Lyudmyla}\snm{Sakhno}}\email{lms@univ.kiev.ua}
\address{Taras Shevchenko National University of Kyiv, 01601 Kyiv, Ukraine}

\markboth{L. Sakhno}{Asymptotics for functionals of powers of a~periodogram}

\begin{abstract}
We present large sample properties and conditions for asymptotic
normality of linear functionals of powers of the periodogram
constructed with the use of tapered data.
\end{abstract}

\received{22 November 2014}
\accepted{8 January 2015}
\publishedonline{4 February 2015}
\end{frontmatter}

\section{Introduction}\label{sec1}

Consider a real-valued measurable zero-mean strictly stationary process
$Y(t)$, $t\in Z$, obeying the following assumption.

\begin{assumption}\label{assump1}
$E|Y(t)|^{k}<\infty$ for all $k$, and $Y(t)$ has (cumulant)
spectral \mbox{densities} of orders $k=2,3,\dots$, that is, there exist the
functions $f_{k}(\lambda _{1},\ldots,\lambda_{k-1})\in
L_{1}(\varLambda^{k-1})$, $\varLambda =(-\pi ,\pi]$, $k=2,3,\ldots$,
such that the cumulant function of order~$k$ is given by
\begin{equation*}
c_{k}(t_{1},\ldots,t_{k-1})=\int
_{\varLambda^{k-1}}f_{k}(\lambda_{1},\ldots,\lambda
_{k-1})e^{i\Sigma_{1}^{k-1}\lambda
_{j}t_{j}}
\,d\lambda_{1}\ldots d\lambda_{k-1}.
\end{equation*}
\end{assumption}

Suppose that we are given the observations $\{Y(t),t{\in} K_{T}\}$, where
$K_{T}=\{-T,\ldots,\allowbreak T\}$, $T\in Z$.

In this paper, we will study large sample properties of the empirical spectral
functionals of the form
%
\begin{equation}
J_{k,T}(\varphi) =\int_{\varLambda}\varphi(\lambda)
I_{T}^{k}(\lambda)\,d\lambda\label{jkt}
\end{equation}
for appropriate functions $\varphi(\lambda)$ with
$\varphi(\lambda)f_{2}^{k}(\lambda)\in L_{1}(\varLambda)$, where $I_{T}^{k}(\lambda)$
is the $k$th power of the periodogram based on the tapered data
$\{h_{T}(t)Y(t),t\in K_{T}\}$, and $k$ is\vadjust{\eject} a positive integer. The taper
will be of the form $h_{T}(t)=h(\frac{t}{T})$ with $h$
satisfying the following standard assumption.

\def\theassumption{H}
\begin{assumption}
\label{assumpH} The function $h(t),t\in R$,
is an even positive
function of bounded variation with bounded support: $h(t)=0$ for $|t|>1$.
\end{assumption}

The periodogram corresponding to the tapered data is defined as

\begin{equation*}
I_{T}(\lambda) =\frac{1}{2\pi H_{2,T}(0)} \bigl|d_{T}(\lambda)
\bigr|^{2},\quad\lambda\in\varLambda, \label{2.3w}
\end{equation*}
where $d_{T}(\lambda)$ is the finite Fourier transform based
on tapered data:
\begin{equation*}
d_{T}(\lambda) =d_{T}^{h}(\lambda) =\sum
_{t\in
K_{T}}e^{-i\lambda t}h_{T}(t)Y(t),\quad
\lambda\in\varLambda, \label{2.2q}
\end{equation*}%
\begin{equation*}%
H_{k,T}(\lambda) =\int_{K_T }h_{T}^{k}(t)e^{-i\lambda t}\,dt,
\end{equation*}
and we suppose that $H_{2,T}(0)\neq0$.

Functionals of the form \eqref{jkt} for $k=1$ have been
extensively studied in the literature, in particular, due to their
applications for parameter estimation in the spectral domain:
their behavior as $T\rightarrow\infty$ is important for
establishing asymptotic properties of so-called minimum contrast
estimators such as Whittle and Ibragimov estimators (see, e.g.,
\cite{LS,AnhLS2007b}, and references therein). The case of
the squared periodogram was treated, for example, in
\cite{DeoChen}, with application to a goodness-of-fit statistics,
and in \cite{Sakhno}, with application to minimum contrast
estimation.

Asymptotic results for the functionals of the form \eqref{jkt}
with general $k\geq2$ were studied in \cite{Chiu} and applied
to derivation of properties of weighted least squares estimators
in the frequency domain and also in \cite{McElroy}, with several
applications discussed, in particular, a frequency domain
goodness-of-fit testing.

In this paper, we derive asymptotic results for functionals~\eqref{jkt} with general $k\geq2$ in a more general setting,
using the tapered data, and under a~different set of conditions
in comparison with those used in \cite{Chiu} and \cite{McElroy};
in the Gaussian case, we state our results in terms of integrability
conditions for the spectral density and weight function.
Methods for the proofs are similar to those used in \cite{Sakhno}
with appropriate modifications required for the more general case
under consideration in the present paper.

\vspace*{3pt}
\section{Results and discussion}\label{sec2}
\vspace*{3pt}

We begin with the following assumptions.

\def\theassumption{\arabic{assumption}}
\setcounter{assumption}{1}

\begin{assumption}\label{assump2} The spectral densities $f_{k}(\lambda_{1},\ldots,\lambda_{k-1})$, $k=2,3,\dots$, of the stochastic
process $Y(t)$ are bounded and continuous.
\end{assumption}
\begin{assumption}\label{assump3}
The weight function $\varphi(\lambda)$ is bounded and continuous.
\end{assumption}

In what follows, we denote the second-order spectral density
$f_{2}(\lambda) $ simply by $f(\lambda)$ omitting the
subscript 2.

\begin{theorem}\label{thm1}
Let Assumptions \ref{assump1}, \ref{assump2}, and \ref{assumpH} hold,
and let the functions
$\varphi,\varphi_{1}(\lambda),\dots,\allowbreak\varphi_{k}(\lambda)$
satisfy Assumption \ref{assump3}. Then, as
$T\rightarrow\infty$,
\begin{enumerate}[leftmargin=0em,itemindent=1.75em,label={\rm(\arabic*)}]
\item\label{thm1.1}\mbox{~}\vspace*{-4.4mm}
\begin{equation*}
EJ_{k,T}(\varphi) \rightarrow k!\int_{\varLambda}\varphi
(\lambda) f^{k}(\lambda)\, d\lambda;
\end{equation*}
\item\label{thm1.2}\mbox{~}\vspace*{-7mm}
\begin{align*}
&\mathit{cov} \bigl(J_{k,T}(\varphi_{1}),J_{l,T}(\varphi_{2})\bigr)
\\
&\quad{}
\sim\frac{2\pi
}{T}e(h)klk!l! \biggl[ \int_{\varLambda}\varphi
_{1}(\lambda) \bigl[\overline{{\varphi}_{2}}(\lambda) +
\overline{{\varphi}_{2}}(-\lambda) \bigr]f^{k+l}(\lambda)\,d
\lambda
\\
&\qquad{}+ \int_{\varLambda^{2}}\varphi_{1}(\lambda
_{1})\overline{\varphi_{2}}(\lambda_{2})f^{k-1}(\lambda_{1})f^{l-1}(\lambda
_{2})f_{4}(\lambda_{1},-\lambda
_{1},\lambda_{2})\,d\lambda_{1}\,d\lambda
_{2} \biggr] ,
\end{align*}
where
\begin{equation*}
e(h)= \biggl\{ \int h^{2}(t)\,dt \biggr\} ^{-2}\int
h^{4}(t)\,dt;
\end{equation*}
\item\label{thm1.3}\mbox{~}\vspace*{-4.4mm}
\[
\mathit{cum}\bigl(J_{m_{1},T}(\varphi_{1}),\ldots,J_{m_{k},T}(\varphi_{k})\bigr) =O\bigl(T^{1-k}\bigr).
\]
\end{enumerate}
\end{theorem}

Suppose now that the process $Y(t)$ is Gaussian. In this case the
above asymptotic results can be stated under the conditions of integrability
on the weight function and spectral density.

\begin{theorem}\label{thm2}
Let $Y(t)$, $t\in Z$, be a Gaussian stationary process with spectral
density $f(\lambda)$, $\lambda\in\varLambda$, such that $f(\lambda
) \in L_{p}(\varLambda)$, and let the functions $\varphi,\varphi
_{1},\ldots,\varphi_{k}\in L_{q}(\varLambda)$, where $1\leq p,q\leq
\infty$.
Suppose also that Assumption~\ref{assumpH} holds.
\begin{enumerate}[leftmargin=0em,itemindent=1.75em,label={\rm(\arabic*)}]
\item\label{thm2.1}
If $p$ and $q$ satisfy the relation
\begin{equation*}
\frac{1}{q}+k\frac{1}{p}=1,
\end{equation*}
then, as $T\rightarrow\infty$,
\begin{equation*}
EJ_{k,T}(\varphi) \rightarrow k!\int_{\varLambda}\varphi
(\lambda) f^{k}(\lambda) \,d\lambda.
\end{equation*}

\item\label{thm2.2}
If $p$ and $q$ satisfy the relation
\begin{equation*}
\frac{1}{q}+\frac{k+l}{2}\cdot\frac{1}{p}=\frac{1}{2},
\end{equation*}
then, as $T\rightarrow\infty$,
\begin{eqnarray*}
&&\mathit{cov} \bigl(J_{k,T}(\varphi_{1}),J_{l,T}(\varphi_{2})\bigr)
\\
&&\quad{} \sim\frac{2\pi
}{T}e(h)klk!l!\int_{\varLambda}\varphi
_{1}(\lambda) \bigl[ \overline{\varphi_{2}}(\lambda) +
\overline{\varphi_{2}}(-\lambda) \bigr] f^{k+l}(\lambda)\,d
\lambda.
\end{eqnarray*}

\item\label{thm2.3}
If $p$ and $q$ satisfy the relation
\begin{equation*}
\frac{1}{q}+k\frac{1}{p}=\frac{1}{2},
\end{equation*}

then the cumulants of orders $r\geq3$ of the normalized functionals
$J_{k,T}(\varphi_{i})$ tend to zero as $T\rightarrow\infty$:
\begin{equation*}
\mathit{cum} \bigl(T^{1/2}J_{k,T}(\varphi_{1}),
\ldots,T^{1/2}J_{k,T}(\varphi_{r})\bigr)
\rightarrow0.
\end{equation*}

\item\label{thm2.4} If $p$ and $q$ satisfy the relation
\begin{equation*}
\frac{1}{q}+\frac{k_{1}+\cdots+k_{r}}{r}\cdot\frac{1}{p}=\frac{1}{2},
\end{equation*}

then the cumulant of $r$th order, $r\geq3$, of the normalized
functionals $J_{k_{i},T}(\varphi_{i})$, $i=1,\ldots,r$, tends to zero as
$T\rightarrow\infty$:
\begin{equation*}
\mathit{cum} \bigl(T^{1/2}J_{k_{1},T}(\varphi_{1}),
\ldots,T^{1/2}J_{k_{r},T}(\varphi_{r})\bigr)
\rightarrow0.
\end{equation*}
\end{enumerate}
\end{theorem}

As corollaries of the above theorems, we obtain the next asymptotic normality
results.

Let us fix the weight functions $\varphi_{1},\ldots,\varphi_{m}$ and denote
\begin{equation*}
J_{T}= \bigl\{J_{k,T}(\varphi_{i}) \bigr
\}_{i=1,\ldots,m}\quad\mbox{\rm and}\quad \tilde{J}= \bigl\{
\tilde{J}_{k}(\varphi_{i}) \bigr\}_{i=1,\ldots,m},
\label{J}
\end{equation*}
where
$\tilde{J}_{k}(\varphi)=k!\int_{\varLambda}\varphi(\lambda) f^{k}(\lambda)\,d\lambda$.

Let $\xi=\{\xi_{i}\}_{i=1,\ldots,m}$ be a Gaussian random vector with zero
mean and second-order moments
\begin{align*}
w_{ij}
&{}=E\xi_{i}\bar{\xi_{j}}\\
&{}=2\pi e(h)(kk!) ^{2} \biggl(\int_{\varLambda}\varphi_{i}(\lambda) \bigl[\bar{\varphi}_{j}(\lambda) +\bar{
\varphi}_{j}(-\lambda) \bigr]f^{2k}(\lambda) \,d\lambda
\notag
\\
&\quad{}+\int_{\varLambda^{2}}\varphi_{i}(\lambda _{1}
) \bar{\varphi}_{j}(\lambda_{2})f^{k-1}(\lambda_{1})f^{k-1}(\lambda _{2})f_{4}(\lambda_{1},-\lambda _{1},
\lambda_{2})\,d\lambda_{1}\,d\lambda _{2} \biggr).
\label{2.7}
\end{align*}

\begin{assumption}\label{assump4}
The spectral density of the second
order $f(\lambda)$, the weight function $\varphi(\lambda)$, and the
taper $h$ are such that $T^{1/2}(EJ_{k,T}(\varphi)-\tilde{J_k}(\varphi
))\rightarrow0$.
\end{assumption}

\begin{theorem}\label{thm3}
Let Assumptions \ref{assump1}, \ref{assump2}, and \ref{assumpH} hold, and let the functions $\varphi_{i}$,
$i=1,\ldots,m
$, satisfy Assumption \ref{assump3}. Then
\begin{equation*}
T^{1/2}(J_{T}-EJ_{T})\overset{D} {\rightarrow}
\xi\quad\text{as }T\rightarrow\infty; \label{2.7a}
\end{equation*}
moreover, if Assumption \ref{assump4} holds for the functions $\varphi_{i}$,
$i=1,\ldots,m$, then
\begin{equation*}
T^{1/2}(J_{T}-\tilde{J})\overset{D} {\rightarrow}\xi
\quad \text{as }T\rightarrow\infty. \label{2.7b}
\end{equation*}
\end{theorem}

Let $\zeta=\{\zeta_{i}\}_{i=1,\ldots,m}$ be a Gaussian random vector with
zero mean and second-order moments
\begin{equation*}
v_{ij}=E\zeta_{i}\bar{\zeta_{j}}=2\pi e(h)(kk!
) ^{2}\int_{\varLambda}\varphi_{i}(\lambda)
\bigl[\bar{\varphi}_{j}(\lambda) +\bar{ \varphi}_{j}(-
\lambda) \bigr]f^{2k}(\lambda) \,d\lambda.
\end{equation*}

\begin{theorem}\label{thm4}
Let $Y(t)$, $t\in Z$, be a Gaussian stationary process with spectral
density $f(\lambda)\in L_{p}(\varLambda)$, and let the functions
$\varphi
_{1},\ldots,\varphi_{m}\in L_{q}(\varLambda)$, where $1\leq p,q\leq
\infty$, be
such that $\frac{1}{q}+k\frac{1}{p}=\frac{1}{2}$. Suppose also that
Assumption~\ref{assumpH} holds. Then
\begin{equation*}
T^{1/2}(J_{T}-EJ_{T})\overset{D} {\rightarrow}
\zeta\quad \text{as }T\rightarrow\infty;
\end{equation*}
moreover, if Assumption~\ref{assump4} holds for the functions $\varphi_{i}$,
$i=1,\ldots,m$, then
\begin{equation*}
T^{1/2}(J_{T}-\tilde{J})\overset{D} {\rightarrow}\zeta
\quad\text{as }T\rightarrow\infty.
\end{equation*}
\end{theorem}

We next state more general results, the joint asymptotic normality for
functionals of different powers of the periodogram.

Consider
\begin{equation*}
J_{k_{1},\ldots,k_{m},T}= \bigl\{J_{k_{i},T}(\varphi_{i}) \bigr\}_{i=1,\ldots,m}
\quad\mbox{\rm and}\quad
\tilde{J}_{k_{1},\ldots,k_{m}}= \bigl\{
\tilde{J}_{k_{i}}(\varphi_{i}) \bigr\}_{i=1,\ldots,m},
\end{equation*}
where $\tilde{J}_{k}(\varphi)=k!\int_{\varLambda}\varphi(\lambda) f^{k}(\lambda)\,d\lambda$.

Let $\tilde{\xi}=\{\tilde{\xi}_{i}\}_{i=1,\ldots,m}$ be a Gaussian random
vector with zero mean and second-order moments
\begin{align*}
\tilde{w}_{ij}
&{}=E\tilde{\xi}_{i}\overline{\tilde{\xi}_{j}}\\
&{}=2\pi e(h)k_{i}k_{i}!k_{j}k_{j}!
\biggl(\int_{\varLambda}\varphi_{i}(\lambda) \bigl[
\bar{\varphi}_{j}(\lambda) +\bar{\varphi}_{j}(-\lambda)\bigr]f^{k_{i}+k_{j}}(\lambda) \,d\lambda\notag
\\
&\quad{}+\int_{\varLambda^{2}}\varphi_{i}(\lambda _{1}
) \bar{\varphi}_{j}(\lambda_{2})f^{k_{i}-1}(\lambda_{1})f^{k_{j}-1}(\lambda _{2})f_{4}(\lambda_{1},-\lambda _{1},
\lambda_{2})\,d\lambda_{1}\,d\lambda _{2} \biggr).
\end{align*}

\begin{theorem}\label{thm5}
Let Assumptions \ref{assump1}, \ref{assump2}, and \ref{assumpH} hold, and let the functions $\varphi_{i}$,
$i=1,\ldots,m$, satisfy Assumption \ref{assump3}. Then
\begin{equation*}
T^{1/2}(J_{k_{1},\ldots,k_{m},T}-EJ_{k_{1},\ldots,k_{m},T})\overset{D} {\rightarrow}
\xi\quad \text{as }T\rightarrow\infty;
\end{equation*}
moreover, if Assumption~\ref{assump4} holds with $k=k_{i}$, $\varphi= \varphi_{i}$, $i=1,\ldots,m$, then
\begin{equation*}
T^{1/2}(J_{k_{1},\ldots,k_{m},T}-\tilde{J}_{k_{1},\ldots,k_{m}})\overset{D} {
\rightarrow}\xi\quad \text{as }T\rightarrow\infty.
\end{equation*}
\end{theorem}

Let $\tilde{\zeta} =\{\tilde{\zeta} _{i}\}_{i=1,\ldots,m}$ be a
Gaussian random
vector with zero mean and second-order moments
\begin{equation*}
\tilde{v}_{ij}=E\tilde{\zeta} _{i}\overline{\tilde{\zeta}
_{j}}=2\pi e(h)k_{i}k_{i}!k_{j}k_{j}!
\int_{\varLambda}\varphi_{i}(\lambda) \bigl[\bar{
\varphi}_{j}(\lambda) +\bar{\varphi}_{j}(-\lambda)\bigr]f^{k_{i}+k_{j}}(\lambda) \,d\lambda.
\end{equation*}

\begin{theorem}\label{thm6}
Let $Y(t)$, $t\in Z$, be a Gaussian stationary process with a spectral
density $f(\lambda)\in L_{p}(\varLambda)$, and let the functions
$\varphi
_{1},\ldots,\varphi_{m}\in L_{q}(\varLambda)$, where $1\leq p,q\leq
\infty$, be
such that $\frac{1}{q}+\min\{k_{i}\}\frac{1}{p}=\frac{1}{2}$. Suppose also
that Assumption~\ref{assumpH} holds. Then
\begin{equation*}
T^{1/2}(J_{k_{1},\ldots,k_{m},T}-EJ_{k_{1},\ldots,k_{m},T})\overset{D} {\rightarrow}
\tilde{\zeta}\quad \text{as }T\rightarrow\infty;
\end{equation*}
moreover, if Assumption~\ref{assump4} holds with $k=k_{i}$, $\varphi= \varphi_{i}$, $i=1,\ldots,m$, then
\begin{equation*}
T^{1/2}(J_{k_{1},\ldots,k_{m},T}-\tilde{J}_{k_{1},\ldots,k_{m}})\overset{D} {
\rightarrow}\tilde{\zeta}\quad \text{as }T\rightarrow\infty.
\end{equation*}
\end{theorem}

\begin{remark}
Integrals of nonlinear functions of the periodogram (including, in
particular, powers of positive orders of the periodogram) were
studied, for example, in {\rm\cite{Taniguchi80}} for discrete time
processes under the assumption of boundedness of the spectral
density. In {\rm\cite{DeoChen}}, the integral functionals of the squared
periodogram were studied for stationary Gaussian series given by
the moving-average representation, and the asymptotic normality result
was stated under the particular assumption of summability of the
coefficients of the representation and continuity of the
derivative of the spectral density. In {\rm\cite{Chiu}} and
{\rm\cite{McElroy}}, the asymptotic results for functionals of powers of
the periodogram of general order have been studied under
the conditions of summability of cumulants of the process. In this
paper, we state the results for discrete-time non-Gaussian
processes under the condition of boundedness of spectral densities
of all orders $($which are supposed to exist$)$, and we also derive the
results for Gaussian case under the conditions of integrability of
the spectral density and weight function.\looseness=1
\end{remark}

\begin{remark}
Conditions on the spectral density under which Assumption~\ref{assump4} will
be satisfied can be formulated analogously to the corresponding
conditions in {\rm\cite{AnhLS2007b}} for the case where $h(t)\equiv1$
and analogously to the conditions in {\rm\cite{Bias}} for the general
$h(t)$ of Assumption~\ref{assumpH}.
\end{remark}

\begin{remark}
The results on asymptotic properties of the integrals of the
powers of the periodogram can be useful for some problems of
statistical inference. One possible application is hypothesis
testing concerning the form of the spectral density of the
process. For example, in {\rm\cite{DeoChen}}, a quadratic
goodness-of-fit test in the spectral domain was studied for Gaussian
processes. Note that the asymptotic normality result for the
corresponding test statistic stated in {\rm\cite{DeoChen}} can be also
derived from our Theorem~\ref{thm6}, that is, under a different set of
conditions. More applications of the integrals of the powers of
the periodogram for goodness-of-fit testing, peak testing, and
assessing model misspecification are presented in {\rm\cite{McElroy}}.
In {\rm\cite{Chiu}} and {\rm\cite{Sakhno}}, the integrals of the squared
periodogram were applied for parametric estimation in the
spectral domain.
\end{remark}
%

\section{Proofs}\label{sec3}

For the proofs of the results of Section~\ref{sec2}, we use the technique
based the properties of the multidimensional kernels of Fej\'{e}r
type (see, e.g., \cite{Bentkus1976,AnhLS2007b}, and
references therein) and the H\"{o}lder--Young--Brascamp--Lieb
inequality (see \cite{AvramLS2010,AvramLS2012}, and
references therein; see also \cite{Sakhno}). In what follows, we
will refer to the latter as the HYBL inequality. The application of
these tools leads to very transparent and elegant proofs. We will
also use the formula giving expressions for cumulants of products
of random variables via products of cumulants of the individual
variables (see, e.g., \cite{LeonovShiryaev,Bentkus1976})
and the multilinearity property of cumulants. The lines of reasonings
are very close to those used in~\cite{Sakhno}.

\begin{proof}[Proof of Theorem \ref{thm1}]
Consider
\begin{align*}
EI_{T}^{k}(\lambda)
&{}=\frac{1}{(2\pi H_{2,T}(0))^{k}}E \bigl[(\mathit{cum}
\bigl(d_{T}(\lambda)d_{T}(-\lambda) \bigr)^{k}\bigr]
\\
&{}=\frac{1}{(2\pi H_{2,T}(0))^{k}}E \bigl[\mathit{cum}(d_{T}(\lambda)d_{T}(-
\lambda)\cdots \mathit{cum}(d_{T}(\lambda)d_{T}(-\lambda) \bigr].
\end{align*}
We apply now the formula for cumulants of products of random variables (see,
e.g., \cite{LeonovShiryaev}); it is convenient to assign the indices to
$\lambda$s in the following way: we can enumerate all $\lambda$s
appearing in the above row from $1$ to $2k$, having in mind that $\lambda_{i}$ with odd indices are simply equal to $\lambda$,
whereas $\lambda_{i}$ with even $i$ are equal to
$-\lambda$. Then we can write down the expectation in the
following form:
\begin{align}
EI_{T}^{k}(\lambda)
&{}=\frac{1}{(2\pi H_{2,T}(0))^{k}}
\sum_{
\substack{
\nu=(\nu_{1},\ldots,\nu_{1}) \\
\text{partition of }(1,\ldots,2k)}
}
\prod_{l=1}^{p}\mathit{cum}
\bigl(d_{T}(\lambda_{i}),i\in\nu_{l} \bigr)
\nonumber\\
&\quad{}\times\prod_{i=1}^{k}\delta(\lambda
_{2i-1}-\lambda) \prod_{i=1}^{k}
\delta(\lambda_{2i}+\lambda). \label{eitk}
\end{align}

The cumulants of the finite Fourier transforms $d_{T}(\lambda)$,
$\lambda
\in\varLambda$, can be written as follows:
\begin{align*}
&\mathit{cum} \bigl(d_{T}(\alpha_{1}),\ldots,d_{T}(\alpha_{k})\bigr)\nonumber\\
&\quad{}=\int_{K_{T}^{k}}\prod_{i=1}^{k}h_{T}(t_{i})e^{-i\Sigma_{1}^{k}\alpha_{j}t_{j}}\;
c_{k}(t_{1}-t_{k},\ldots,t_{k-1}-t_{k})\,dt_{1}\ldots dt_{k}\nonumber\\
&\quad{}=\int_{\varLambda^{k-1}}f_{k}(\gamma_{1},\ldots,\gamma_{k-1} )\nonumber\\
&\qquad{}\times\prod_{j=1}^{k-1}H_{1,T}(\gamma_{j}-\alpha_{j})H_{1,T} \Biggl(-\sum
_{j=1}^{k-1}\gamma_{j}-\alpha_{k} \Biggr) \,d\gamma_{1}\ldots d\gamma_{k-1},
\end{align*}
where
\begin{equation*}
H_{1,T}(\lambda) =\int_{K_{T}}h_{T}(t)e^{-it\lambda}\,dt.
\end{equation*}

Correspondingly, we obtain the following formula for the expectation
of~$J_{k,T}(\varphi)$:
\begin{align}
&EJ_{k,T}(\varphi)=E\int_{\varLambda}\varphi(\lambda)
I_{T}^{k}(\lambda) \,d\lambda\notag
\\
&\quad{}=\int_{\varLambda}\varphi(\lambda) \frac{1}{(2\pi
H_{2,T}(0))^{k}}\sum
_{\substack{ \nu=(\nu_{1},\ldots,\nu_{p}) \\
\text{partition of }(1,\ldots,2k)}}\int_{{\varLambda}^{2k-p}}\prod
_{i=1}^{p}f_{\llvert
\nu_{i}\rrvert }(\gamma_{j},j
\in\tilde{\nu}_{i})\notag
\\
&\qquad{}\times\prod_{j=1}^{2k}H_{1,T}(\gamma_{j}-\lambda_{j})\prod_{l=1}^{p}
\delta \biggl(\sum_{j\in\nu_{l}}\gamma_{j} \biggr)
\prod_{i=1}^{k}\delta(\lambda_{2i-1}-
\lambda) \prod_{i=1}^{k}\delta(\lambda
_{2i}+\lambda)\, d\gamma^{\prime
}\,d\lambda. \label{ejkt}
\end{align}

Here and in similar formulas below, we use the following notation:
for a~set of natural numbers $\nu$, we denote by $\llvert \nu
\rrvert $ the number of elements in $\nu$ and by
$\widetilde{\nu}$ the subset of $\nu$ that contains
all elements of $\nu$
except the last one. Integration in the inner integral in the above formula
is understood with respect to $(2k-p)$-dimensional vector $\gamma^{\prime}$ obtained from the vector
$\gamma=(\gamma_{1},\ldots,\gamma_{2k})$ due to $p$ restrictions on
the variables $\gamma_{j}$, $j=1,\ldots,2k$, described by the
Kronecker delta functions $\delta$.

Now we note that the products $\prod_{j=1}^{k}H_{1,T}(\lambda
_{j})$ in the case where $\sum_{j=1}^{k}\lambda_{j}=0$ give rise to
a class of $\delta$-type kernels (or Fej\'{e}r-type kernels). Namely, if
Assumption~\ref{assumpH} holds and $H_{k,T}(0)\neq0$, then
%
\begin{equation}
\varPhi_{k,T}^{h}(\lambda_{1},\ldots,\lambda
_{k-1}):=\frac{1}{(2\pi)^{k-1}H_{k,T}(0)}\prod_{j=1}^{k-1}H_{1,T}
(\lambda_{j})H_{1,T} \Biggl(-\sum
_{j=1}^{k-1}\lambda_{j} \Biggr)\label{kernel}
\end{equation}
is a kernel over $\varLambda^{k-1}$, which is an approximate
identity for convolution (see, e.g., \cite{Dahlhaus}), and
%
\begin{align}
&\underset{T\rightarrow\infty} {\lim}\int_{{\varLambda
}^{k-1}}G(u_{1}-v_{1},\ldots,u_{k-1}-v_{k-1})\varPhi_{k,T}^{h}(u_{1},\ldots,u_{k-1})
\,du_{1}\ldots du_{k-1}\notag
\\
&\quad{}=G(v_{1},\ldots,v_{k-1}), \label{convol}
\end{align}
provided that the function $G(\cdot,\ldots,\cdot)$ is
bounded and continuous at the point $(v_{1},\ldots,\allowbreak v_{k-1})$.

The asymptotic behavior of the right-hand side of (\ref{ejkt}) can be
evaluated basing on the property (\ref{convol}).

Let us first consider the partitions $\nu$ composed by pairs. For those partitions, when the products of only the
cumulants of the form $\mathit{cum}(d_{T}(\lambda),d_{T}(-\lambda))$
appear in~(\ref{eitk}), we obtain under the integral sign in
(\ref{ejkt}) the terms of the form
\begin{align*}
&\frac{1}{(2\pi H_{2,T}(0))^{k}} \biggl\{ \int f(\gamma)H_{1,T}(\gamma -\lambda
)H_{1,T}(-\gamma+\lambda)\,d\gamma \biggr\} ^{k}\\
&\quad{}= \biggl\{ \int f(\gamma)\varPhi_{2,T}^{h}(\gamma-\lambda )\,d
\gamma \biggr\} ^{k}.
\end{align*}
We note that there are $k!$ such terms, therefore, in the
expression for $EJ_{k,T}(\varphi)$, we have the term
%
\begin{equation}
k!\int_{\varLambda}\varphi(\lambda) \biggl\{ \int
_{\varLambda
}f(\gamma)\varPhi_{2,T}^{h}(\gamma-
\lambda)d\gamma \biggr\} ^{k}\,d\lambda, \label{main}
\end{equation}
and this is the only case where we have $k$ kernels, and all $k$
factors $\frac{1}{2\pi H_{2,T}(0)}$ are used to compose these kernels $\varPhi
_{2,T}^{h}(\cdot)$.

In all other partitions, we will be able to compose from 1 to $k-1$ kernels
taking combination of $H_{1,T}(\cdot)$ with suitable arguments: for each
2nd-order kernel, we will use one of the factors $\frac{1}{2\pi
H_{2,T}(0)}$ from $\frac{1}{(2\pi H_{2,T}(0))^{k}}$; otherwise, when
composing the $l$th order kernel with $l\neq2$, we will need the
normalizing factor $\frac{1}{(2\pi)^{l-1}H_{l,T}(0)}$, and
therefore we will modify the factor $\frac{1}{(2\pi H_{2,T}(0))^{k}}$
by taking, instead, $\frac{(2\pi)^{l-1}H_{l,T}(0)}{(2\pi H_{2,T}(0))^{k}}$.

So, for those partitions, when we compose kernels of orders, say,
$l_{1},\ldots,l_{r}$, with $\sum_{i=1}^{r}l_{i}=2k$, the corresponding
integral in (\ref{ejkt}) will be represented in the form of a generalized
convolution of some product of spectral densities of different
orders with the product of kernels of orders $l_{1},\ldots,l_{r}$,
and the factor
%
\begin{equation}
\frac{\prod_{i=1}^{r}(2\pi)^{l_{i}-1}H_{l_{i},T}(0)}{(2\pi
H_{2,T}(0))^{k}} \label{factor}
\end{equation}
will be supplied to the integral. For example, in the simplest case
where $p=1 $, the corresponding term in (\ref{ejkt}) can be represented
as follows:
\begin{align*}
&\frac{(2\pi)^{2k-1}H_{2k,T}(0)}{(2\pi H_{2,T}(0))^{k}}%
\int_{\varLambda}\varphi(\lambda)
\int_{\varLambda
^{2k-1}}f_{2k}(\gamma_{1},\ldots,\gamma
_{k-1})
\\
&\quad{}\times\varPhi_{2k,T}^{h}(\gamma_{1}-\lambda
_{1},\ldots,\gamma_{2k-1}-\lambda_{2k-1})
\\
&\quad{}\times\prod_{i=1}^{k}\delta(\lambda
_{2i-1}-\lambda) \prod_{i=1}^{k}
\delta(\lambda_{2i}+\lambda) \prod_{i=1}^{2k-1}d
\gamma_{i}\,d\lambda.
\end{align*}

Now we take into account the following asymptotics for $H_{k,T}(0)$: $%
H_{k,T}(0)\sim TH_{k}(0)$, where $H_{k}(0)=\int h^{k}(\lambda)\,d\lambda$, and conclude that, in the case of $r$ kernels, $1\leq r\leq
k-1$, the factor (\ref{factor}) is asymptotically of order $\frac
{1}{T^{k-r}}%
$; the corresponding integrals containing these kernels will converge to
finite limits under the conditions of the theorem according to (\ref{convol}).
Therefore, the expectation is obtained as the limit of~(\ref{main}). This
gives statement \ref{thm1.1} of the theorem.

Consider
\begin{align}
&\mathit{cov} \bigl(J_{k,T}(\varphi_{1}),J_{l,T}(\varphi
_{2}) \bigr)\notag\\
&\quad{} =\frac{1}{(2\pi
H_{2,T}(0))^{k+l}}\notag
\\
&\qquad{}
\times\int_{\varLambda^{2}}
\varphi_{1}(\alpha)
\overline{\varphi_{2}}(\beta)
\mathit{cum}\bigl(
\bigl(d_{T}(\alpha)d_{T}(-\alpha) \bigr)^{k},
\bigl(d_{T}(\beta)d_{T}(-\beta) \bigr)^{l}
\bigr)
\,d\alpha \,d\beta. \label{covar}
\end{align}
The cumulant under the integral sign in (\ref{covar}) according to the
formula for calculation of cumulants of products of random variables
can be
written in the form
%
\begin{equation}
\sum_{\nu=(\nu_{1},\ldots,\nu_{p})}\prod_{i=1}^{p}\mathit{cum}
\bigl(d_{T}(\mu_{j}),\mu_{j}\in\nu
_{i} \bigr), \label{sumcum}
\end{equation}
where the summation is taken over all indecomposable partitions $\nu= (\nu_{1},\ldots,\nu_{p})$, $|\nu_{i}|>1$, of the table $T_{2}$ with
two rows, $\{\alpha,-\alpha,\ldots,\alpha,-\alpha\}$ (of length
$2k$) and $%
\{\beta,-\beta,\ldots,\beta,-\beta\}$ (of length $2l$). For
asymptotic analysis of expression~(\ref{covar}), we can use the
reasonings analogous to those for the case of functionals of
squared periodogram in \cite{Sakhno}, but now, dealing with the
tapered case, we need to keep track of normalizing factors for appearing
kernels. Again, similarly to the previous consideration of
the expectation, we analyze all possible partitions and
kernels that can be composed for every particular partition. Let us
first consider the terms in (\ref{sumcum}) that correspond to
partitions by pairs, that is,
%
\begin{equation}
\prod_{i=1}^{k+l}\mathit{cum} \bigl(d_{T}(\mu_{i}),d_{T}(\lambda_{i}) \bigr),
\label{cum4}
\end{equation}
where $\mu_{i},\lambda_{i}\in\{\alpha,-\alpha,\beta,-\beta\}$,
and $%
\nu=\{(\mu_{i},\lambda_{i}),i=1,\ldots,k+l\}$ forms an~indecomposable
partition of the table $T_{2}$.

In the case where we have the $k-1$ cumulants $\mathit{cum}(d_{T}(\alpha
),d_{T}(-\alpha))$ and $l-1$ cumulants $\mathit{cum}(d_{T}(\beta),d_{T}(-\beta))$
in the product~\eqref{cum4}, according to formula~(\ref{kernel}), we
can compose $%
k+l-1$ kernels ($k+l-2$ kernels of order 2 and one of order~4), and
the factor before the integral in (\ref{covar}) becomes of the form
$\frac{%
2\pi H_{4,T}(0)}{(H_{2,T}(0))^{2}}$, which is asymptotically of order
$\frac{%
1}{T}$. Only the terms of this kind in (\ref{sumcum}) give the main
contribution (of order $\frac{1}{T}$) into the covariance (\ref
{covar}); all
other terms in (\ref{cum4}) produce a smaller-order contribution to the
covariance (\ref{covar}%
). More precisely, in order to describe the asymptotics of
the covariance, we have to consider, among the terms in (\ref{sumcum}), the
following ones:
%
\begin{align}
&\bigl(\mathit{cum} \bigl(d_{T}(\alpha),d_{T}(-\alpha) \bigr)
\bigr) ^{k-1} \bigl(\mathit{cum} \bigl(d_{T}(\beta),d_{T}(-
\beta) \bigr) \bigr) ^{l-1}\notag\\
&\quad{}\times{} \bigl[ \mathit{cum} \bigl(d_{T}(\alpha),d_{T}(\beta)
\bigr)\mathit{cum} \bigl(d_{T}(-\alpha),d_{T}(-\beta) \bigr)\notag
\\
&\quad{}+\mathit{cum} \bigl(d_{T}(\alpha),d_{T}(-\beta) \bigr)\mathit{cum}
\bigl(d_{T}(-\alpha),d_{T}(\beta) \bigr) \bigr]. \label{cummain}
\end{align}
Their contribution to the covariance is of the form
\begin{align*}
&\frac{1}{(2\pi H_{2,T}(0))^{k+l}}\int\int\varphi_{1}(\alpha) \overline{%
\varphi_{2}}(\beta)\\
&\qquad{}
\times \biggl[
\int f(\gamma_{1})H_{1,T}(\gamma_{1}-\alpha)H_{1,T}(-\gamma_{1}+\alpha)\,d\gamma_{1}
\biggr] ^{k-1}
\\
&\qquad{}\times \biggl[
\int f(\gamma_{2})H_{1,T}(\gamma_{2}-\beta)H_{1,T}(-\gamma_{2}+\beta)\,d\gamma_{2}
\biggr] ^{l-1}
\\
&\qquad{}\times \biggl[
\int f(\gamma_{3})H_{1,T}(\gamma_{3}-\alpha )H_{1,T}(-\gamma_{3}-\beta)\,d\gamma_{3}
\\
&\qquad{}\times\int f(\gamma_{4})H_{1,T}(\gamma_{4}+
\alpha )H_{1,T}(-\gamma _{4}+\beta)\,d\gamma_{4}
\\
&\qquad{}+\int f(\gamma_{3})H_{1,T}(\gamma_{3}-
\alpha)H_{1,T}(-\gamma _{3}+\beta )\,d\gamma_{3}
\\
&\qquad{}\times \int f(\gamma_{4})H_{1,T}(\gamma_{4}+
\alpha )H_{1,T}(-\gamma_{4}-\beta)\,d\gamma_{4}
\biggr] \,d\alpha \,d\beta
\\
&\quad{}=\frac{2\pi H_{4,T}(0)}{(H_{2,T}(0))^{2}}\int\int\varphi _{1}(\alpha)%
\overline{\varphi_{2}}(\beta) \biggl[ \int f(\gamma_{1})
\varPhi _{2,T}^{h}(\gamma_{1}-\alpha)\,d
\gamma_{1} \biggr] ^{k-1}
\\
&\qquad{}\times
\biggl[ \int f(\gamma_{2})\varPhi_{2,T}^{h}(\gamma_{2}-\beta )\,d\gamma _{2} \biggr] ^{l-1}
\\
&\qquad{}\times\int\int f(\gamma_{3})f(\gamma_{4}) \bigl[\varPhi
_{4,T}^{h}(\gamma _{3}-\alpha,-
\gamma_{3}-\beta,\gamma_{4}+\alpha)
\\
&\qquad{}+\varPhi_{4,T}^{h}(\gamma_{3}-\alpha,-
\gamma_{3}+\beta,\gamma _{4}+\alpha)%
\bigr]\,d
\gamma_{3}\,d\gamma_{4}\,d\alpha \,d\beta.
\end{align*}
Taking into account formula (\ref{convol}), we can evaluate the latter as
\begin{equation*}
\sim\frac{2\pi}{T}\frac{\int h^{4}(t)\,dt}{(\int h^{2}(t)\,dt)^{2}%
}\int_{\varLambda}\varphi
_{1}(\lambda) \bigl[\overline{\varphi_{2}}(\lambda) +
\overline{\varphi_{2}}(-\lambda) \bigr] f^{k+l}(\lambda) \,d\lambda\quad \mbox{\rm as}\ T\rightarrow\infty.
\end{equation*}
We note also that there are $klk!l!$ terms of the form (\ref{cummain})
in (%
\ref{sumcum}).

Let us now consider the terms in (\ref{sumcum}) that correspond to other
partitions. We can see that one more possibility is left to compose $k+l-1$
kernels, namely, the case of the terms of the form
\begin{align*}
&
\bigl[
\mathit{cum} \bigl(d_{T}(\alpha),d_{T}(-\alpha) \bigr)
\bigr]^{k-1}
\bigl[
\mathit{cum} \bigl(d_{T}(\beta),d_{T}(-\beta) \bigr)
\bigr]^{l-1}
\\
&\quad{}\times
\mathit{cum} \bigl(d_{T}(\alpha),d_{T}(-\alpha),d_{T}(\beta),d_{T}(-\beta) \bigr)
\end{align*}
in the sum (\ref{sumcum}) (there are $klk!l!$ such terms), and the
corresponding contribution to the covariance (\ref{covar}) is of
the form
\begin{align*}
&\frac{1}{(2\pi H_{2,T}(0))^{k+l}}\int\int\varphi_{1}(\alpha) \overline{%
\varphi_{2}}(\beta)
\\
&\qquad{}\times \biggl[ \int f(\gamma_{1})H_{1,T}(\gamma
_{1}-\alpha)H_{1,T}(-\gamma_{1}+\alpha)\,d\gamma
_{1} \biggr] ^{k-1}
\\
&\qquad{}\times \biggl[ \int f(\gamma_{2})H_{1,T}(\gamma
_{2}-\beta)H_{1,T}(-\gamma_{2}+\beta)\,d\gamma
_{2}%
 \biggr] ^{l-1}\\
&\qquad{}\times\int\int\int f_{4}(\mu_{1},\mu_{2},
\mu_{3})H_{1,T}(\mu_{1}-\alpha)H_{1,T}(\mu_{2}+\alpha)
\\
&\qquad{}\times H_{1,T}(\mu_{3}-\beta)H_{1}^{T}%
 \Biggl(-\sum_{i=1}^{3}\mu_{i}+
\beta \Biggr)\,d\mu_{1}\,d\mu_{2}\,d\mu_{3}\,d\alpha\,d\beta
\\
&\quad{}=\frac{2\pi H_{4,T}(0)}{(H_{2,T}(0))^{2}}\int\int\varphi_{1}(\alpha) 
\overline{
\varphi_{2}}(\beta)\\
&\qquad{}\times\int f(\gamma_{1})\varPhi
_{2,T}^{h}(\gamma_{1}-\alpha)\,d\gamma
_{1}\int f(\gamma_{2})\varPhi_{2,T}^{h}(\gamma_{2}-\beta)\,d\gamma_{2}
\\
&\qquad{}\times\int\int\int f_{4}(\mu_{1},\mu_{2},
\mu_{3})\varPhi_{4,T}^{h}(\mu_{1}-
\alpha,\mu_{2}+\alpha,\mu_{3}-\beta)
\,d\mu _{1}\,d\mu_{2}\,d\mu_{3}\,d\alpha\,d\beta
\\
&\quad{}\sim\frac{2\pi}{T}\frac{\int h^{4}(t)\,dt}{(\int
h^{2}(t)\,dt)^{2}}\int\int\varphi_{1}(\alpha)
\overline{\varphi_{2}}%
(\beta) f(\alpha) f(\beta)
f_{4}(\alpha,-\alpha,\beta)\,d\alpha \,d\beta\quad\! \mbox{\rm as}\,T
\rightarrow\infty.
\end{align*}
In all other cases, we can compose less than $k+l-1$ kernels; the corresponding
integrals will converge to finite limits supplied by the factor of
orders not exceeding~$\frac{1}{T^{2}}$.

Summarizing the above reasonings, we come to the asymptotics for the
covariance as given in statement \ref{thm1.2} of the theorem.

We now evaluate the asymptotic behavior of the cumulant of order $k\geq 3$:
\begin{align*}
&\mathit{cum} \bigl(J_{m_{1},T}(\varphi_{1}),\ldots,J_{m_{k},T}
(\varphi_{k})\bigr)\\
&\quad{} =\frac{1}{(2\pi
H_{2,T}(0))^{M}}\int_{\varLambda^{k}}
\varphi_{1}(\alpha_{1})\cdots\varphi_{k}(\alpha_{k})
\\
&\qquad{}\times \mathit{cum} \bigl(\bigl(d_{T}(\alpha_{1})d_{T}(-
\alpha_{1}) \bigr)^{m_{1}},\ldots, \bigl(d_{T}(\alpha _{k})d_{T}(-\alpha_{k})
\bigr)^{m_{k}} \bigr)\,d\alpha_{1}\ldots d\alpha_{k},
\end{align*}
where $M=\sum_{i=1}^{k}m_{i}$.

The cumulant under the integral sign can be represented as the sum
%
\begin{equation}
\sum_{\nu=(\nu_{1},\ldots,\nu_{p})}\prod_{i=1}^{p}\mathit{cum}
\bigl(d_{T}(\mu_{j}),\mu_{j}\in\nu
_{i} \bigr), \label{sumcumk}
\end{equation}
where the summation now is taken over all indecomposable partitions $\nu
=(\nu_{1},\ldots,\allowbreak\nu_{p})$, $|\nu_{i}|>1$, of the table $T_{k}$
with $k$ rows $\{\alpha_{i},-\alpha_{i},\ldots,\alpha_{i},-\alpha
_{i}\}$, $%
i=1,\ldots,k$, the~length of the $i$th row being $2m_{i}$. Starting
again with
consideration of partitions by pairs, we can see that with these partitions
we can compose at most $\sum_{i=1}^{k}(m_{i}-1)+1=M-k+1$ kernels ($M-k$
kernels $\varPhi_{2,T}^{h}$ and one kernel $\varPhi_{4,T}^{h}$), and
the corresponding
integrals will converge to finite limits supplied with the factor $\frac
{%
(2\pi)^{3}H_{4,T}(0)}{(2\pi H_{2,T}(0))^{M-(M-k)}}=\frac{%
(2\pi)^{3}H_{4,T}(0)}{(2\pi H_{2,T}(0))^{k}}$, which is
asymptotically of order $\frac{1}{T^{k-1}}$. With all other partitions, we
will be able to compose no more that $M-k+1$ kernels; therefore, their
contribution to cumulant (\ref{sumcumk}) will be of order less than $%
\frac{1}{T^{k-1}}$. This gives statement \ref{thm1.3} of the theorem.
\end{proof}

\begin{proof}[Proof of Theorem \ref{thm2}]
We use the same calculations as those in the proof of Theorem~\ref{thm1},
but to analyze the limit behavior of the integrals representing
the cumulants, we will appeal to the HYBL inequality (see
\cite{AvramLS2010,AvramLS2012,Sakhno}). The
reasonings follow the same lines as in \cite{Sakhno}, so here we just
point out the key steps.

For the Gaussian case, we will have only partitions by pairs in (\ref{ejkt}):
\begin{align}
&EJ_{k,T}(\varphi) =E\int_{\varLambda}\varphi(\lambda)I_{T}^{k}(\lambda) \,d\lambda\notag\\
&\quad{}=\int_{\varLambda}\varphi(\lambda) \frac{1}{(2\pi
H_{2,T}(0))^{k}}\sum
_{\substack{ \nu=(\nu_{1},\ldots,\nu
_{k}),|\nu
_{k}|=2, \\
\text{partition of }(1,\ldots,2k)}}\int_{{\varLambda}%
^{k}}\prod
_{i=1}^{k}f(\gamma_{j},j\in\tilde{
\nu}_{i}) \notag
\\
&\qquad{}\times\prod_{j=1}^{2k}H_{1,T}(\gamma_{j}-\lambda_{j})\prod_{l=1}^{k}
\delta \biggl(\sum_{j\in\nu_{l}}\gamma_{j} \biggr)
\prod_{i=1}^{k}\delta(\lambda_{2i-1}-
\lambda) \prod_{i=1}^{k}\delta(\lambda
_{2i}+\lambda)\, d\gamma^{\prime
}\,d\lambda.\label{ejktgauss}
\end{align}

We consider separately the term (\ref{main}):
\begin{eqnarray*}
&&k!\int_{\varLambda}\varphi(\lambda) \biggl\{ \int
_{\varLambda}f(\gamma)\varPhi_{2,T}^{h}(\gamma-
\lambda)d\gamma \biggr\} ^{k}\,d\lambda
\\
&&\quad{}=k!\int_{\varLambda^{k}} \Biggl[ \int_{\varLambda}\varphi(\lambda) \prod_{j=1}^{k}f(\gamma
_{j}-\lambda)\,d\lambda \Biggr] \prod_{j=1}^{k}
\varPhi_{2,T}^{h}(\gamma_{j})\prod
_{j=1}^{k}d\gamma_{j}.
\end{eqnarray*}
Note that the convergence to the finite limit $k!\int_{\varLambda
}\varphi(\lambda)f^{k}(\lambda)\,d\lambda$ will be
assured if we assume the conditions for statement \ref{thm2.1} of the
theorem.

Now consider the remaining terms: we have the integrals over
$\varLambda^{k+1}$ with integrands composed by products of the
functions $\varphi$ with $k$ functions $f$ and $2k$ functions
$H_{1,T}$ with some linear relations between the arguments of
these functions; these integrals are supplied with the factor
$\frac{1}{(2\pi H_{2,T}(0))^{k}}$.

Applying the HYBL inequality, we can bound each such integral by the
expression\vspace*{6pt}
%
\begin{equation}
\frac{1}{(2\pi H_{2,T}(0))^{k}}const\Vert\varphi\Vert_{q}\Vert f\Vert
_{p}^{k}\Vert H_{1,T}\Vert_{r}^{2k},
\label{bound1}
\end{equation}
provided that $\varphi(\lambda)\in L_{q}(\varLambda)$, $f(\lambda)\in L_{p}(\varLambda)$, and $H_{1,T}(\lambda)\in
L_{r}(\varLambda)$ with
\begin{equation*}
\frac{1}{q}+k\frac{1}{p}+2k\frac{1}{r}=k+1.
\end{equation*}

If we choose $r=2$ and take into account that, under Assumption~\ref{assumpH}, we
have $\Vert H_{1,T}\Vert_{r}\leq CT^{1-1/r}$ and $H_{1,T}(0)\sim T$,
then from (\ref{bound1}) we arrive at the bound $const\Vert\varphi
\Vert_{q}\Vert f\Vert_{p}^{k}$ as $T\rightarrow\infty$, with
the restrictions on $p$ and $q$ as in statement~\ref{thm2.1} of the
theorem. From this point we can repeat the same arguments as in
\cite{Sakhno} to show that, in fact, this bound can be strengthen
to $o(1)$ as $T\rightarrow\infty$.

The similar reasonings are applied to derive statements \ref{thm2.2}--\ref{thm2.4}
of Theorem~\ref{thm2}.
\end{proof}

\end{document}